\documentclass{article}

\usepackage[utf8]{inputenc}	
\usepackage{amsmath}
\usepackage{amsfonts}
\usepackage{amsthm}
\usepackage{amssymb}
\usepackage{textcomp}
\usepackage{gensymb}
\usepackage{graphicx}
\usepackage{float}
\usepackage{wrapfig}
\usepackage{hyperref}
\usepackage{dsfont}
\usepackage{bbm}
\usepackage[all]{xy}
\usepackage[left=25mm, top=20mm, right=10mm, bottom=20mm, nohead, nofoot]{geometry}

\newtheorem{theorem}{Theorem}
\newtheorem*{theorem*}{Theorem}
\newtheorem{lemma}[theorem]{Lemma}
\newtheorem{conjecture}[theorem]{Conjecture}

\newtheorem{remark}[theorem]{Remark}

\DeclareMathOperator{\Con}{Con}

\DeclareMathOperator{\Aff}{Aff}

\begin{document}

\title{Hardness of almost embedding simplicial complexes in $\mathbb{R}^d$, II
\footnote{I would like to thank S. Avvakumov for helpful discussion.}
}

\author{Emil Alkin
}

\date{}

\maketitle

\begin{abstract}
A map $f: K \to \mathbb{R}^d$ of a simplicial complex is an \textit{almost embedding} if $f(\sigma) \cap f(\tau) = \varnothing$ whenever $\sigma, \tau$ are disjoint simplices of $K$.
Fix integers $d,k \geqslant 2$ such that $k+2 \leqslant d \leqslant\frac{3k}2+1$.
Assuming that the ``preimage of a cycle is a cycle'' (Conjecture \ref{con:pre})
we prove $\mathbf{NP}$-hardness of the algorithmic problem of recognition of almost embeddability of finite $k$-dimensional complexes in $\mathbb{R}^d$.
Assuming that $\mathbf{P} \ne \mathbf{NP}$ (and that the ``preimage of a cycle is a cycle'') we prove that the embedding obstruction is incomplete for $k$-dimensional complexes in $\mathbb{R}^d$ using configuration spaces.
Our proof generalizes the Skopenkov-Tancer proof of this result for $d = \frac{3k}{2} + 1$.  
\end{abstract}

\tableofcontents

\section{Introduction}
Let $K$ be a finite simplicial complex.
A map $f : K \to \mathbb{R}^d$ is an \textit{almost embedding} if $f(\sigma) \cap f(\tau) = \varnothing$ whenever $\sigma, \tau$ are disjoint simplices of $K$. 

Almost embeddings naturally appear in studies of embeddings. See more motivations in \cite[\S1, `Motivation and background' part]{ST} and \cite[Remark 5.7.4]{Algor}.

\begin{theorem}\label{t:nphard}
Assume that Conjecture~\ref{con:pre} is true.
The algorithmic problem of recognition of almost embeddability of finite $k$-dimensional complexes in $\mathbb{R}^d$ is $\mathbf{NP}$-hard for $d,k \geqslant 2$ such that $k+2 \leqslant d \leqslant\frac{3k}2+1$.
\end{theorem}

The (simplicial) \textit{deleted product} of $K$ is
$$\widetilde{K} := \cup \{\sigma \times \tau : \sigma, \tau \text{ are simplices of } K, \sigma \cap \tau = \varnothing \};$$
i.e., $\widetilde{K}$ is the union of products $\sigma \times \tau$ formed by disjoint simplices of $K$.

A map $\widetilde{f} : \widetilde{K} \to S^{d-1}$ is \textit{equivariant} if $\widetilde{f}(y, x) = -\widetilde{f}(x, y)$ for each pair $(x, y)$ from $\widetilde{K}$.

\begin{theorem}\label{t:eqmap}
Fix integers $d,k \geqslant 2$ such that $k+2 \leqslant d \leqslant\frac{3k}2+1$.
Assume that $\mathbf{P} \ne \mathbf{NP}$. Assume that Conjecture~\ref{con:pre} is true. Then there exists a finite $k$-dimensional complex $K$ that does not admit an almost embedding in $\mathbb{R}^d$ but for which there exists an equivariant map $\widetilde K\to S^{d-1}$.
\end{theorem}

The  particular cases of both Theorem~\ref{t:nphard} and Theorem~\ref{t:eqmap} (without assuming Conjecture~\ref{con:pre}) for $k, d$ such that $d = \frac{3k}{2} + 1$ or $d \equiv 1\ (\text{mod}\ 3)$ and $2 \leqslant k \leqslant d \leqslant \frac{3k}{2} + 1$ are proved in \cite{ST}. 
Those proofs are based on \cite[Singular Borromean Rings Lemma 2.4]{AMSW}. 

Theorem~\ref{t:nphard} is deduced analogously (see details below) to \cite[Theorem 1(b)]{ST} from the following generalized Singular Borromean Rings Lemma~\ref{l:SBR}. 
We prove Lemma~\ref{l:SBR} using Conjecture~\ref{con:pre} stating that ``the preimage of a cycle is a cycle'' (a similar in a sense result can be known in folklore).

Theorem~\ref{t:eqmap} is deduced analogously (see details below) to \cite[Theorem 1(a)]{ST} from Theorem~\ref{t:nphard}.

By $\cdot$ denote some point in $S^n$ for some $n$.

\begin{lemma}[Singular Borromean Rings]\label{l:SBR}
Assume that Conjecture~\ref{con:pre} is true.
For each $k > l \geqslant 1$ let $T := S^l \times S^l$ be the $2l$-dimensional torus with meridian $m := S^l \times \cdot$ and parallel $p := \cdot \times S^l$, and let $S^k_p$ and $S^k_m$ be copies of $S^k$. Then there is no PL map $f : T \sqcup S^k_p \sqcup S^k_m \to \mathbb{R}^{k+l+1}$ satisfying the following three properties:
\begin{enumerate}
    \item the $f$-images of the components are pairwise disjoint;
    \item $fS^k_p$ is linked modulo 2 with $fp$ and is not linked modulo 2 with $fm$, and
    \item $fS^k_m$ is linked modulo 2 with $fm$ and is not linked modulo 2 with $fp$.
\end{enumerate}
\end{lemma}


\begin{remark} 
\begin{itemize}
    \item[(a)] The condition $l \geqslant 1$ is essential in Lemma~\ref{l:SBR}. Indeed, there exists a PL map $f : T \sqcup S^{k}_p \sqcup S^k_m \to S^{k+l+1}$ satisfying the properties of Lemma~\ref{l:SBR} for $k \geqslant 1,\ l = 0$.
    
    Let $m:=\{\pm1\} \times \{1\}$ and $p:=\{1\} \times \{\pm1\}$.
    
    By $\mathds{1}_{n, m}$ denote the point in $\mathbb{R}^m$ with $n$-th coordinate equals to one and the others equal to zeros.
    
    Define the map $f: \underbrace{ \{\pm1\} \times \{\pm1\} }_{=T} \sqcup S^k_p \sqcup S^k_m \to \mathbb{R}^{k+1}$ by the rule
    
    \begin{equation*}
    f(x) = 
    \begin{cases}
    x \times {(0)}^{k-1} &\text{if $x \in T$}\\
    x + \mathds{1}_{2, k+1} - \mathds{1}_{1, k+1} &\text{if $x \in S_m^k$}\\
    x + \mathds{1}_{1, k+1} - \mathds{1}_{2, k+1} &\text{if $x \in S_p^k.$}
    \end{cases}
    \end{equation*}
    
    It can be easily checked that the map $f$ satisfies the properties.
    
    
    \item[(b)] The condition $k > l$ is essential in Lemma~\ref{l:SBR}. Indeed, there exists a PL map $f : T \sqcup S^{k}_p \sqcup S^k_m \to S^{k+l+1}$ satisfying the properties of Lemma~\ref{l:SBR} for $l = k$.  

    Define the map $f: T \sqcup S^k_p \sqcup S^k_m \to \underbrace{S^{k} \times D^{k+1} \cup D^{k+1} \times S^{k}}_{=S^{2k+1}}$ by the rule

    \begin{equation*}
    f(x) = 
    \begin{cases}
    x &\text{if $x \in T$}\\
    (0, x) &\text{if $x \in S_m^k$}\\
    (x, 0) &\text{if $x \in S_p^k.$}
    \end{cases}
    \end{equation*}
    Clearly, the map $f$ satisfies the first property from Lemma \ref{l:SBR}. Since $\left|f S_p^k \cap \Con (fp)\right| = \left|S^k \times 0 \cap \cdot \times D^{k+1}\right| = \left|\{(\cdot, 0)\}\right| = 1$ and $f S_p^k \cap \Con (fm) = S^k \times 0 \cap D^{k+1} \times \cdot  = \varnothing $, the map $f$ satisfies the second property. Analogously, the map $f$ satisfies the third property.
\end{itemize}

\end{remark}

\section{Proofs}

\begin{proof}[Proof of Theorem~\ref{t:nphard}] 
Formally, Theorem \ref{t:nphard} follows by modified version of \cite[Theorem 2]{ST} obtained by substitution the hypothesis ``$d = \frac{3k}{2} + 1$" with ``$k+2 \leqslant d \leqslant \frac{3k}{2} + 1$".

The proof of the modified version of \cite[Theorem 2]{ST} is obtained from the proof of \cite[Theorem 2]{ST} by:
\begin{itemize}
    \item setting ``$l := d - k - 1$'';
    \item changing the second sentence of the second paragraph of `construction of $K(\Phi)$' to ``Take a triangulation of $2l$-torus $T$ extending triangulations of its meridian and parallel $a$ and $b$ as boundaries of $(l + 1)$-simplices.'';
    \item adding the sentence ``Since $k \geqslant 2l$, $K(\Phi)$ is a $k$-complex.'' after the second paragraph of `construction of $K(\Phi)$';
    \item changing the last sentence of the `only if' part to ``Since $k \geqslant 2l$, all this contradicts the Singular Borromean Rings Lemma~\ref{l:SBR} applied to the restriction of $f$ to $S_q \sqcup S_r \sqcup T_{qr}$".
\end{itemize}
\end{proof}

\begin{proof}[Proof of Theorem \ref{t:eqmap}]
Theorem \ref{t:eqmap} follows by Theorem \ref{t:nphard} and the existence
of a polynomial algorithm for checking the existence of equivariant maps \cite{CKV}. Indeed, for fixed $d, k$ it is polynomial time decidable whether there exists an equivariant map $\widetilde{K} \to S^{d-1}$ \cite{CKV}. Given that almost embeddabilty implies the existence of an equivariant map, Theorem \ref{t:nphard} implies Theorem \ref{t:eqmap}.
\end{proof}

Recall some known 
definitions.


A \textit{$c$-chain} in $\mathbb{R}^d$  is a finite set $C$ of $c$-simplices in $\mathbb{R}^d$. By $V(C)$ denote the set of vertices of simplices from $C$.

A $c$-chain $C$ in $\mathbb{R}^d$ is called \textit{simplicial} if the intersection of any two $c$-simplices in $C$ is a face of both of them whenever their intersection is not empty.

The \textit{boundary} of a $c$-chain in $\mathbb{R}^d$ is the set of those $(c-1)$-simplices in $\mathbb{R}^d$ that are faces of an odd number of the chain's simplices.

A chain whose boundary is empty is called a \textit{cycle}.

For a set $V = \{ v_i \}_{i=1}^n$ of points in $\mathbb{R}^d$ denote $\left\{ \sum\limits_{i = 1}^n \alpha_i v_i \left|\ \sum\limits_{i = 1}^n \alpha_i = 1 \right. \right\}$ by $\Aff(V)$.

A set $V$ of points in $\mathbb{R}^d$ is \textit{in strong general position} \cite{PS} if for any collection $\{ V_1, V_2, \ldots , V_r \}$ of $r$ pairwise disjoint subsets of $V$ the following holds ($\dim \varnothing := -\infty$):

$$ \dim \bigcap\limits^r_{i = 1} \Aff(V_i) \leqslant \sum\limits^r_{i=1} \dim \Aff(V_i) - d(r - 1).$$


A set $W$ of points in $\mathbb{R}^d$ is \textit{in strong general position with respect to} a $c$-chain $C$ in $\mathbb{R}^d$ if $W \cap V(C) = \varnothing$ and the set $W \sqcup V(C)$ is in strong general position.

For any PL manifold $N$ a PL map $f: N \to \mathbb{R}^d$ is \textit{in strong general position with respect to} a $c$-chain $C$ in $\mathbb{R}^d$ if there exists a triangulation $T_N$ of $N$ such that 
\begin{itemize}
    \item the map $f$ is linear on each simplex of $T_N$, and
    \item the set of images of vertices of $T_N$ is in strong general position with respect to the chain $C$.
\end{itemize}

\begin{conjecture}[preimage of a cycle is a cycle]\label{con:pre}
Let $c,n,d$ be non-negative integers such that $c < d$.
For any $c$-chain $C$ in $\mathbb{R}^d$, $n$-dimensional closed PL manifold $N$, and PL map $f: N \to \mathbb{R}^d-\partial C$ in strong general position with respect to $C$, the preimage $f^{-1}C$ is the support of a $(c+n-d)$-cycle in $|T_N| \subset \mathbb{R}^m$.
\end{conjecture}

Notice that the original proof of \cite[Singular Borromean Rings Lemma 2.4]{AMSW} implicitly used a simpler version of Conjecture \ref{con:pre} with additional hypothesis that the chain $C$ avoids self-intersection points of $f(N)$.


\begin{proof}[Proof of Lemma~\ref{l:SBR} (see the full proof in the appendix)] 
It is sufficient to make the following changes in proof of \cite[Lemma 2.4]{AMSW}:
\begin{itemize}
    \item the second part of the sentence after $(**)$ should be changed to ``they are cycles by Conjecture~\ref{con:pre} and because  $f(S_p^n) \cap \partial (C_T) = f(S_p^n) \cap f(T) = \varnothing = f(S_p^n) \cap f(S_m^n) = f(S_p^n) \cap \partial (C_m)$.";
    
    \item the second sentence in the text after $(****)$ should be changed to ``By Conjecture~\ref{con:pre} and $f(T) \cap \partial (C_p) = f(T) \cap f(S_p^n) = \varnothing$, it follows that the $l$-chain $f_T^{-1}C_p$ is a cycle in $T$."
\end{itemize}
\end{proof}

\section{Proof of Conjecture~\ref{con:pre}}

For Conjecture~\ref{con:pre} we need 
Lemma~\ref{eq} and Conjecture~\ref{con:simpl}.

Let us recall some known definitions.

The intersection of finite number of open half-spaces and an $n$-hyperplane is called an {\it open $n$-polytope} if this intersection is bounded and non-empty.
The closure of an open $n$-polytope is called an {\it $n$-polytope}.

Let
\begin{itemize}
    \item $\sigma, \tau$ be polytopes in $\mathbb{R}^d$;
    \item $[\tau : \sigma] \in \mathbb{Z}_2$ be the characteristic function of `$\sigma$ is a face of $\tau$';
    \item $[\sigma \subset \partial\tau] \in \mathbb{Z}_2$ be the characteristic function of `$\sigma \subset \partial\tau$';
    \item $P$ be a set of polytopes in $\mathbb{R}^d$;
    \item $[P:\sigma] := \sum\limits_{\tau \in P} [\tau:\sigma]$;
    \item $[P:\sigma]^{\text{inc}} := \sum\limits_{\tau \in P} [\sigma \subset \partial\tau]$.
\end{itemize}



\begin{lemma}\label{eq}
Let $P$ be a finite set of $c$-polytopes in $\mathbb{R}^n$. 
Suppose that
\begin{itemize}
    \item[1.] for any two different polytopes $\sigma_1, \sigma_2$ from $P$ the intersection $\sigma_1 \cap \sigma_2$ is a polytope of dimension at most $c-1$ and is contained in $\partial \sigma_1 \cap \partial \sigma_2$;
    
    \item[2.] $[P:\sigma]^{\text{inc}} = 0$ for any $(c-1)$-polytope $\sigma$ in $\mathbb{R}^n$ . 
\end{itemize}

Then the union of $P$ is the support of a simplicial $c$-cycle in $\mathbb{R}^n$. 
\end{lemma}

\begin{proof} [Proof]
We prove the simpler version of Lemma~\ref{eq} in which all ``polytopes'' are changed to ``simplices'' and the final sentence are changed to ``Then $P$ is a simplicial $c$-cycle in $\mathbb{R}^n$.''. The original lemma can easily be reduced to the simpler version by triangulating polytopes.

Denote by $\partial P$ the finite set of $(c-1)$-simplices $\sigma$ in $\mathbb{R}^n$ for which 

(*) the number of simplices $\tau \in P$ such that $\sigma$ is a simplex of $\partial\tau$ is odd ($\Longleftrightarrow [P:\sigma] = 1$).

Assume the contrary, i.e., that the set $\partial P$ is non-empty.

Let $\sigma$ be a maximal element of the set $\partial P$ ordered by the inclusion. 
Then $\sigma$ is a $(c-1)$-simplex such that (*) holds. 
Let $L := \left\{ \sigma' \ \left|\ \sigma \varsubsetneq \sigma',\ \sigma' \text{ is a $(c-1)$-simplex of } \tau \text{ for some } \tau \in P \right. \right\}$.

Then
$$[P:\sigma]^{\text{inc}} = \sum\limits_{\tau \in P} [\sigma \subset \partial\tau] = \sum\limits_{\sigma' \in \{\sigma\} \cup L} [P:\sigma'] = \underbrace{[P:\sigma]}_{=1} +  \sum\limits_{\sigma' \in L} \underbrace{[P:\sigma']}_{=0} = 1.$$
Hence (2) does not hold.
This contradiction concludes the proof.
\end{proof}

A set $V$ of points in $\mathbb{R}^d$ is \textit{in general position} if for any $i \in [d-1]$ the following holds:

there is no $i$-hyperplane containing at least $i+2$ points from $V$.

A set $W$ of points in $\mathbb{R}^d$ is \textit{in general position with respect to} a $c$-chain $C$ in $\mathbb{R}^d$ if for any $c$-simplex $\sigma \in C$ the set $W$ does not contain vertices of $\sigma$ and the set of vertices of $\sigma$ and points from $W$ is in general position.

For any PL manifold $N$ a PL map $f: N \to \mathbb{R}^d$ is \textit{in general position with respect to} a $c$-chain $C$ in $\mathbb{R}^d$ if there exists a triangulation $T_N$ of $N$ such that 
\begin{itemize}
    \item the map $f$ is linear on each simplex of $T_N$, and
    \item the set of images of vertices of $T_N$ is in general position with respect to the chain $C$.
\end{itemize}

\begin{conjecture}\label{con:simpl}
Let $C$ be a $c$-chain in $\mathbb{R}^d$ with $c < d$. Let $U$ be a finite set of points in $\mathbb{R}^d$ in strong general position with respect to $C$. Then the union of $C$ is the support of a simplicial $c$-chain $C'$ in $\mathbb{R}^d$ such that $U$ is in general position with respect to $C'$.
\end{conjecture}

\begin{remark}
The condition of strong general position is essential in Conjecture~\ref{con:simpl}, i.e. modified version of Conjecture~\ref{con:simpl} obtained by substituting ``strong general position'' with ``general position'' is wrong. It can be shown by the following counterexample.

Let $c := 1,\ d := 2$. By $S$ denote a unit circle in $\mathbb{R}^2$  centered at the origin $O$. Let $xy, zt, uv$ be different diameters of $S$. Let $C := \{xy, zt \}$, $U := \{ u, v \}$. Obviously, for any simplicial $1$-chain $C'$ with $\bigcup C' = \bigcup C$ the point $O$ is a vertex of $C'$. Even though points $x, y, z, t, u, v$ are in general position, the set $U \cup \{ O \}$ is not in general position.
\end{remark}

\begin{proof}[Proof of Conjecture~\ref{con:pre}] 

Let $T_N$ be a triangulation of $N$ satisfying the conditions from definition of strong general position with respect to $C$. 

By Conjecture~\ref{con:simpl}, there exists a simplicial $c$-chain $C'$ in $\mathbb{R}^d$ such that $\bigcup C = \bigcup C'$ and the map $f$ is in general positon with respect to $C'$.

For any $c$-simplex $\sigma$ in $C'$ and any $n$-simplex $\gamma$ in $T_N$ the intersection $\gamma \cap f^{-1}(\sigma)$ is either an empty set or a $(c+n-d)$-polytope. Then the preimage $f^{-1}(\sigma)$ is the union of a finite set $P_{\sigma}$ of $(c+n-d)$-polytopes in $|T_N| \subset \mathbb{R}^m$. Let $P$ be the disjoint union $\bigsqcup\limits_{\sigma \in C' \text{\ is a $c$-simplex}} P_{\sigma}$.  

In the following three bullet points we prove that
for any two $(c+n-d)$-polytopes $s$ and $t$ from $P$ having a common point their intersection is a polytope of dimension at most $c+n-d-1$ and $s \cap t \subset \partial s \cap \partial t$ holds.
\begin{itemize}
    \item \textit{Both $s$ and $t$ are contained in some $n$-simplex $\gamma$ of $T_N$}. In this case $s \cap t = \left(\gamma \cap f^{-1}(\sigma)\right) \cap \left(\gamma \cap f^{-1}(\tau)\right) = \gamma \cap f^{-1}(\sigma \cap \tau)$ for some $c$-simplices $\sigma$, $\tau$ in $C'$ is a polytope of dimension at most $n+(c-1)-d=c+n-d-1$.
    And, 
    \[
    \gamma \cap f^{-1}(\sigma \cap \tau)
    = \gamma \cap f^{-1}(\partial \sigma \cap \partial \tau)
    \subset
    \left\{
    \begin{array}{cc}
    \gamma \cap f^{-1}(\partial \sigma) \subset \partial s\\
    \gamma \cap f^{-1}(\partial \tau) \subset \partial t
    \end{array}
    \right. \Longrightarrow s \cap t \subset \partial s \cap \partial t.
    \]
    
    \item \textit{The polytopes $s$ and $t$ are contained in some different $n$-simplices $\gamma$ and $\delta$ of $T_N$, respectively.
    And $s, t \in P_{\sigma}$ for some $c$-simplex $\sigma$ in $C'$ }. In this case $s \cap t = \left( \gamma \cap f^{-1}(\sigma) \right) \cap \left(\delta \cap f^{-1}(\sigma)\right) = \gamma \cap \delta \cap f^{-1}(\sigma)$ is a polytope of dimension at most $(n-1)+c-d=c+n-d-1$.
    And, 
    \[
    \gamma \cap \delta \cap f^{-1}(\sigma)
    = \partial \gamma \cap \partial \delta \cap f^{-1}(\sigma)
    \subset
    \left\{
    \begin{array}{cc}
    \partial \gamma \cap f^{-1}(\sigma) \subset \partial s\\
    \partial \delta \cap f^{-1}(\sigma) \subset \partial t
    \end{array}
    \right. \Longrightarrow s \cap t \subset \partial s \cap \partial t.
    \]
    
    \item \textit{The polytopes $s$ and $t$ are contained in some different $n$-simplices $\gamma$ and $\delta$ of $T_N$, respectively. And $s \in P_{\sigma}, t \in P_{\tau}$ for some different $c$-simplices $\sigma, \tau$ in $C'$}. In this case $s \cap t = \left( \gamma \cap f^{-1}(\sigma) \right) \cap \left(\delta \cap f^{-1}(\tau)\right) = \gamma \cap \delta \cap f^{-1}(\sigma \cap \tau)$ is a polytope of dimension at most $(n-1)+(c-1)-d=c+n-d-2$.
    And, 
    \[
    \gamma \cap \delta \cap f^{-1}(\sigma \cap \tau)
    = \partial \gamma \cap \partial \delta \cap f^{-1}(\partial \sigma \cap \partial \tau)
    \subset
    \left\{
    \begin{array}{cc}
    \partial \gamma \cap f^{-1}(\sigma) \subset \partial s\\
    \partial \delta \cap f^{-1}(\tau) \subset \partial t
    \end{array}
    \right. \Longrightarrow s \cap t \subset \partial s \cap \partial t.
    \]
\end{itemize}

By Lemma~\ref{eq}, it suffices to prove that for any $(c+n-d-1)$-polytope $u$ in the boundary of some $(c+n-d)$-polytope in $P$ the number of $(c+n-d)$-polytopes in $P$ containing $u$ is even. Let us consider two cases.
\begin{itemize}
    \item \textit{$u$ is contained in some $(n-1)$-simplex $\eta$ of $T_N$}. There exists only one $c$-simplex $\sigma$ in $C'$ such that $u \subset \partial \left( f^{-1}(\sigma) \cap \gamma \right)$ for some $n$-simplex $\gamma$ of $T_N$ containing $\eta$. By $\delta$ denote another $n$-simplex containing $\eta$. Since $f(u) \cap \sigma \subset f(\delta) \cap \sigma$ and $f(u) \cap \sigma = f(u)$ is not empty, the intersection $\delta \cap f^{-1}(\sigma)$ is $(c+n-d)$-polytope containing $u$. Hence there are only two $(c+n-d)$-polytopes containing $u$.
    \item \textit{$u$ is contained in some $n$-simplex $\gamma$ of $T_N$ and intersects the interior of $\gamma$}. As in the previous case, there exists a $c$-simplex $\sigma$ such that $u \subset \partial \left( f^{-1}(\sigma) \cap \gamma \right)$. By $\kappa$ denote a $(c-1)$-simplex containing $f(u)$. Let $\tau$ be a $c$-simplex  in $C'$. Then $u \subset \gamma \cap f^{-1}(\tau)$ if and only if $\kappa \subset \tau$. Since the number of $c$-simplices containing $\kappa$ is even, we have that the number of $(c+n-d)$-polytopes containing $u$ is even.
\end{itemize}

\end{proof}

Apparently there is another way to prove Conjecture~\ref{con:pre} similar to proofs of \cite[Lemma 11.4]{Hu69} and \cite[Lemma 1]{Hu70}. This way is unlikely to be simpler because it is based on another definition of general position, which entails its own technical difficulties.

\section{Appendix}

\begin{proof}[Proof of the Singular Borromean Rings Lemma~\ref{l:SBR}]
(This proof repeats the proof of \cite[Singular Borromean Rings Lemma 2.4]{AMSW} with minor changes).

Assume to the contrary that the map $f$ exists.
Without loss of generality, we may assume that $f$ is in general position.

We denote by $\partial$ the boundary of a chain.

We can view $f(T)$, $f(S^n_p)$, and $f(S^n_m)$ as $2l$-, $n$- and $n$-dimensional cycles in general position in $\mathbb{R}^{n+l+1}$.
Denote by $C_T$, $C_p$, and $C_m$ singular cones in general position over $f(T)$, $f(S^n_p)$, and $f(S^n_m)$, respectively.
We view these cones as $(2l+1)$-, $(n+1)$- and $(n+1)$-dimensional chains.
The contradiction is
$$0\underset{(1)}=|\partial(C_T \cap C_p \cap C_m)|\ \underset{(2)}=
\ |\underbrace{\partial C_T}_{=f(T)} \cap C_p \cap C_m|\ +
\ |C_T \cap \underbrace{\partial C_p}_{=f(S^n_p)}\cap C_m|\ +
\ |C_T \cap  C_p \cap \underbrace{\partial C_m}_{=f(S^n_m)}|\ \underset{(3)}=1+0+0=1.$$
Here (1) follows because $C_T \cap C_p \cap C_m$ is a $1$-dimensional chain, so its boundary is a set of an even number of points.
Equation (2) is Leibniz formula.
So it remains to prove (3).
\phantom\qedhere
\end{proof}

\begin{proof}[Proof of (3)]
For $X\in \{T,S^n_m,S^n_p\}$ denote $f_X:=f|_X$.

For the {\it second term} we have
$$|C_T \cap f(S^n_p) \cap C_m|
\ \overset{(*)}=\ |(f_{S^n_p}^{-1}C_T)\cap(f_{S^n_p}^{-1}C_m)|\ \overset{(**)}=\ 0,\quad\text{where}$$
\begin{itemize}
\item[(*)] holds because $(n+1)+(2l+1)+2n<3(n+l+1)$, so by general position $C_T\cap C_m$ avoids self-intersection points of $f(S^n_p)$,

\item[(**)] holds by the well-known higher-dimensional analogue of \cite[Parity Lemma 3.2.c]{Sk14} (which is proved analogously) because the intersecting objects are general position cycles in $S^n_p$; they are cycles by Conjecture~\ref{con:pre} and because  $f(S_p^n) \cap \partial (C_T) = f(S_p^n) \cap f(T) = \varnothing = f(S_p^n) \cap f(S_m^n) = f(S_p^n) \cap \partial (C_m)$.
\end{itemize}
Analogously $|C_T \cap  C_p \cap f(S^n_m)|=0$.

For the {\it first term} we have
$$|f(T) \cap C_p \cap C_m|
\ \overset{(***)}=\ |(f_T^{-1}C_p)\cap (f_T^{-1}C_m)|\ \overset{(****)}=\ m\cap p\ =\ 1,\quad\text{where}$$
\begin{itemize}
\item[(***)] holds because $n\ge l\Leftrightarrow 2(n+1)+4l<3(n+l+1)$, so by general position $C_p\cap C_m$ avoids self-intersection points of $f(T)$,

\item[(****)] is proved as follows:

By Conjecture~\ref{con:pre} and $f(T) \cap \partial (C_p) = f(T) \cap f(S_p^n) = \varnothing$, it follows that the $l$-chain $f_T^{-1}C_p$ is a cycle in $T$.
By conditions (b) and (c) of Lemma~\ref{l:SBR} we have
$$|p\cap f_T^{-1}C_p|\ =\ |f(p)\cap C_p|\ =\ 1 \text{\quad and \quad}
|m\cap f_T^{-1}C_p|\ =\ |f(m)\cap C_p|\ =\ 0.$$
I.e. the cycle $f_T^{-1}C_p$  intersects the parallel $p$ and the meridian $m$ at $1$ and $0$ points modulo $2$, respectively.
Therefore $f_T^{-1}C_p$ is homologous to the meridian $m$.
Likewise, $f_T^{-1}C_m$ is homologous to the parallel $p$.
This implies (****).
\end{itemize}
\end{proof}

\bibliographystyle{alpha}
\bibliography{ae}

\end{document}